\documentclass[12pt,a4paper]{amsart}
\allowdisplaybreaks 

\usepackage{color}
\usepackage{amsmath}
\usepackage{amssymb}
\usepackage{amsthm}
\usepackage{amsfonts}
\usepackage{latexsym}
\usepackage{hyperref}
\usepackage{tikz}
\usepackage{enumerate}
\usetikzlibrary{matrix,arrows}

\usepackage{mathtools}
\usepackage{mhchem}
\usepackage{gensymb}
\newcommand{\myrightleftarrows}[1]{\mathrel{\substack{\xrightarrow{#1} \\[-.9ex] \xleftarrow{#1}}}}

\newtheorem{theorem}{Theorem}[section]
\newtheorem{corollary}[theorem]{Corollary}
\newtheorem{lemma}[theorem]{Lemma}
\newtheorem{proposition}[theorem]{Proposition}
\theoremstyle{definition}
\newtheorem*{definition}{Definition}

\numberwithin{equation}{section}

\def \bC {\mathbb C}
\def \bD {\mathbb D}

\def \bN {\mathbb N}

\def \bR {\mathbb R}

\def \bR {\mathbb R}
\def \bR {\mathbb R}
\def \bT {\mathbb T}

\def \bZ {\mathbb Z}

\def \cA {\mathcal A}

\def \cD {\mathcal D}
\def \cE {\mathcal E}

\def \cG {\mathcal G}
\def \cH {\mathcal H}

\def \cK {\mathcal K}
\def \cL {\mathcal L}

\def \cP {\mathcal P}

\def \cS {\mathcal S}

\def \fg {\mathfrak g}

\def \fk {\mathfrak k}

\def \ft {\mathfrak t}

\def \fU {\mathfrak U}

\def \al {\alpha}
\def \la {\lambda}
\def \ph {\varphi}

\def \lan {\langle}
\def \ran {\rangle}
\def \de {\partial}

\def \inv{^{-1}}

\def \tr {\text{\rm tr\,}}

\def\be{\begin{equation}}
\def\ee{\end{equation}}
\def\bes{\begin{equation*}}
\def\ees{\end{equation*}}
\def\bea{\begin{equation}\begin{aligned}}
\def\eea{\end{aligned}\end{equation}}
\def\beas{\begin{equation*}\begin{aligned}}
\def\eeas{\end{aligned}\end{equation*}}

\title[]{On the Schwartz correspondence\\ for Gelfand pairs of polynomial growth}

\dedicatory{In memory of Edoardo Vesentini,\\ who has been at the origin of our interest in Gelfand theory}

\author{Francesca Astengo}
\address{Dipartimento di Matematica, Universit\`a di Genova, Via Dodecaneso 35, 16146 Genova, Italy} 
\email{{\tt astengo@dima.unige.it}}

\author{Bianca Di Blasio}
\address{Dipartimento di Matematica e Applicazioni, Universit\`a di Milano Bicocca, Via Cozzi 53, 20125 Milano, Italy } 
\email{{\tt bianca.diblasio@unimib.it}}

\author{Fulvio Ricci}
\address{Scuola Normale Superiore, Piazza dei Cavalieri
7, 56126 Pisa, Italy } 
\email{{\tt fulvio.ricci@sns.it}}



\begin{document}

\maketitle

\begin{abstract} 
Let $(G,K)$ be a Gelfand pair, with $G$ a Lie group of polynomial growth, and let $\Sigma\subset\bR^\ell$ be a homeomorphic image of the Gelfand spectrum, obtained by choosing a generating system $D_1,\dots,D_\ell$ of $G$-invariant differential operators on $G/K$ and associating to a bounded spherical function $\ph$ the $\ell$-tuple of its eigenvalues under the action of  the $D_j$'s. 

We say that property (S) holds for $(G,K)$ if the spherical transform maps the bi-$K$-invariant  Schwartz space $\cS(K\backslash G/K)$ isomorphically onto $\cS(\Sigma)$, the space of restrictions to $\Sigma$ of the Schwartz functions on $\bR^\ell$. This property is known to hold for many nilpotent pairs, i.e., Gelfand pairs where $G=K\ltimes N$, with $N$ nilpotent.

In this paper we enlarge the scope of this analysis outside the range of nilpotent pairs, stating the basic setting for general pairs  of polynomial growth and then focussing on strong Gelfand pairs.
\end{abstract}

\baselineskip15pt

\section{Introduction}

Within harmonic analysis on general locally compact groups, the theory of Gelfand pairs is the branch that exploits forms of commutativity in a noncommutative context.

If a group $G$ is not abelian, convolution on it is not commutative, but it may happen that it becomes commutative when restricted to functions satisfying appropriate invariance properties.

Most of the relevant examples of this kind can be reconducted to the situation where commutativity holds for $L^1$-functions that are bi-invariant with respect to a compact subgroup $K$ of $G$, i.e., 
\be\label{biK}
f(k_1xk_2)=f(x)\ ,\qquad \forall\, x\in G\ ,\ k_1,k_2\in K\ .
\ee

 In this case one says that $(G,K)$ is a {\it Gelfand pair}.

The importance of this type of condition is that bi-$K$-invariant functions represent the kernels of integral operators on the $G$-homogeneous space $M=G/K$ which commute with the action of $G$.

The theory becomes particularly rich when $G$ is assumed to be a Lie group and $M$ is connected. Then condition \eqref{biK} is equivalent to commutativity of the algebra of $G$-invariant differential operators on $M$. In this case the homogeneous space $M$ is also called a {\it commutative space}. The simplest examples of commutative spaces are spheres, Euclidean and hyperbolic spaces, with the action of their  isometry groups.

The notion of Gelfand pair appeared first in the study of (infinite dimensional) irreducible unitary representations of semisimple Lie groups. Since then it has been widely applied to analysis on semisimple groups and symmetric spaces \cite{H, War}. More recent is the interest in pairs where the group $G$ is not semisimple, with motivations that are both algebraic \cite{BR, Vin, Yak} and analytic \cite{BJR1,BJR2, DS, HulR}.

\medskip

In this paper we are only interested in Gelfand pairs $(G,K)$ with $G$ a Lie group and this condition must be assumed even if not made explicit. The group $K$ is a compact subgroup of $G$ and we will almost always require that $M=G/K$ is connected.

With a standard language abuse, we use the same notation $L^p(G/K)$, $C^\infty(G/K)$ etc., to denote both spaces of functions on the homogeneous space $M$ and   spaces of right-$K$-invariant functions on $G$. In the same way, $L^p(K\backslash G/K)$ etc. will stand both for spaces of $K$-invariant functions on $M$ and for spaces of bi-$K$-invariant functions on $G$.

Spherical analysis on a Gelfand pair $(G,K)$ is the analogue of Fourier analysis on $\bR^n$ or on the torus~$\bT$.\footnote{In fact, Fourier analysis on $\bR^n$ is spherical analysis on the pair $(\bR^n,\{0\})$ and similarly on $\bT$.}
Precisely, it comes from application of Gelfand theory to the commutative convolution algebra $L^1(K\backslash G/K)$.

Our interest is in the mapping properties of the spherical transform when $G$ has polynomial growth. More precisely we want to know if the transform establishes a correspondence between Schwartz spaces (in the sense that will be explained soon), as the classical Fourier transform does on $\bR^n$.

This problem has been investigated, always with positive answers,  on a large number of ``nilpotent pairs'' \cite{ADR1, ADR2, FiR, FRY1,FRY2} and it can be conjectured that the same should hold for general nilpotent pairs. In absence of a general approach able to cover the totality of pairs, progress on this topic goes by gradual bootstrapping from simpler cases to less simple ones.

Here we move in a different direction, opening the problem of validity of the Schwartz correspondence when $G$ has polynomial growth but the pair is not nilpotent.  In this perspective, the first case to analyze is that of {\it strong Gelfand pairs} of polynomial growth, which present a particularly neat structure. As it will be explained in the first paragraph of Section 4, the problem does not make sense in the same terms when $G$ has not polynomial growth.

 The scope of this paper is threefold. It wants to be an introduction to the problem, recalling the basic notions and properties of  Gelfand pairs of general Lie groups, emphasizing the special features of pairs with polynomial growth and of strong Gelfand pairs. It proposes a paradigm for concrete attack of the Schwartz correspondence problem for strong pairs via reduction to $K$-types.\footnote{The method presented here finds application in \cite{ADR3} to the strong Gelfand pair $(U_2\ltimes\bC^2,U_2)$}. Finally, it gives a first application of this paradigm for strong pairs with $K$ abelian.

Along the way, we give a proof of Schwartz correspondence for general compact Gelfand pairs.

\medskip

We explain now in what sense we talk of Schwartz correspondence in the context of Gelfand pairs.

The spherical transform $\cG f$ of a function $f\in L^1(K\backslash G/K)$ is defined on the Gelfand spectrum $\Sigma$, the set of bounded {\it spherical functions}, as
$$
\cG f(\ph)=\int_G f(x)\ph(x\inv)\,dx\ ,\qquad \ph\in\Sigma\ .
$$

The operator $\cG$ has several basic properties in common with the Fourier transform, in particular, it is injective and
$$
\cG(f*g)=(\cG f)(\cG g)\ ,\qquad \cG:L^1(K\backslash G/K)\longrightarrow C_0(\Sigma)\ .
$$

Despite the abstract definition of $\Sigma$ and of its topology, i.e., the weak* topology inherited from its inclusion in $L^\infty(K\backslash G/K)$, it is possible to give concrete realizations of $\Sigma$.

In fact, when realized as a function on $M$, a spherical functions $\ph$ is an eigenfunction of  all $G$-invariant differential operators. In the algebra of such operators we can then fix a finite set of generators $\cD=\{D_1,\dots,D_\ell\}$ and associate to $\ph$ the $\ell$-tuple $\xi=(\xi_1,\dots,\xi_\ell)\in\bC^\ell$ of its eigenvalues under $D_1,\dots,D_\ell$. This $\ell$-tuple identifies $\ph$ uniquely and the map $\ph\longmapsto\xi$ is a homeomorphism with closed image \cite{FR}, denoted by $\Sigma_\cD$. 

If $G$ has polynomial growth, choosing symmetric generators $D_j$, the eigenvalues are real, so that $\Sigma_\cD\subset\bR^\ell$.
Setting 
\be\label{Ssigma}
\cS(\Sigma_\cD)=\big\{g_{|_{\Sigma_\cD}}:g\in\cS(\bR^\ell)\big\} \cong \cS(\bR^\ell)/\{g:g_{|_{\Sigma_\cD}}=0\}\ ,
\ee
endowed with the quotient topology, 
the Schwartz correspondence problem concernes the validity of the following property:
\smallskip

\begin{enumerate}
\item[(S)] $\cG$ maps $\cS(K\backslash G/K)$ onto $\cS(\Sigma_\cD)$ isomorphically.
\end{enumerate}
\smallskip

\bigskip

\section{Gelfand pairs}

Let $G$ be a  Lie group and $K$ a compact subgroup. We introduce the following notation.
\begin{itemize}
\item If $X(G)$ is a space of functions on $G$, we denote by $X(K\backslash G/K)$ the subspace of bi-$K$-invariant functions. 
\item If $K'$ is a group of automorphisms of $G$, $X(G)^{K'}$ denotes the subspace of $K'$-invariant elements of $X(G)$.
\item By ${\rm Int}(K)$ we denote the group of conjugations by elements of $K$. It will be clear from the context whether conjugations are meant on $K$ itself or on $G$. An ${\rm Int}(K)$-invariant function on $G$ is called {\it $K$-central}. Obviously,
\be\label{inclusion}
X(K\backslash G/K)\subset X(G)^{{\rm Int}(K)}\ .
\ee
\item By $\bD(G)^{{\rm Int}(K)}\cong \fU(\fg)^{{\rm Ad}(K)}$ we denote the algebra of left-invariant differential operators on $G$ which are also invariant under ${\rm Int}(K)$. 
\item By $\bD(G/K)\cong  \fU(\fg)^{{\rm Ad}(K)}/\big( \fU(\fg)^{{\rm Ad}(K)}\cap\fU(\fg)\fk\big)$ we denote the algebra of $G$-invariant differential operators on $G/K$.
\end{itemize}

We refer to  \cite{Far, vanD, W} for undefined terms, unproven statements and further details.

\begin{definition}
The pair $(G,K)$ is called a {\it Gelfand pair} if either of the following equivalent conditions is satisfied:
\begin{enumerate}
\item[(i)] the convolution algebra $L^1(K\backslash G/K)$ is commutative;
\item[(ii)] if $\pi\in\widehat G$ contains the trivial representation of $K$, it does so without multiplicity.
\end{enumerate}
Assuming $G/K$ connected, the above conditions are also equivalent to
\begin{enumerate}
\item[(iii)] $\bD(G/K)$ is commutative.
\end{enumerate}

Notice that it is not necessary to specify which Haar measure is being used in (i) to define~$L^1$, because in either case commutativity of $L^1(K\backslash G/K)$ implies unimodularity of $G$.

\end{definition}

\subsection{$K$-invariant functions on a group $H$}\qquad

A large class of Gelfand pairs arise in the following context.

Let $K$ is a compact group  of automorphisms of another Lie group $H$, and set $G=K\ltimes H$, with product
\be\label{product}
(k,h)(k',h')=\big(kk',h(k\cdot h')\big)\ ,
\ee
where $(k,h)\longmapsto k\cdot h$ denotes the action of $k\in K$ on $H$.

We identify the quotient space $G/K$ with $H$, via the correspondence $h\longmapsto hK=\{(k,h):k\in K\}$.
For $f$ defined on $H$, we call $\widetilde f$ its right-$K$-invariant lifting to $G$, $\widetilde f(k,h)=f(h)$. Then $\widetilde f$ is bi-$K$-invariant if and only if $f$ is $K$-invariant. In particular, under this correspondence,
\be\label{L^1(H)}
L^1(K\backslash G/K)\cong L^1(H)^K
\ee
as convolution algebras.
In a similar way one establishes the isomorphism 
\be\label{D(H)}
\bD(G/K)\cong \bD(H)^K
\ee 
as operator algebras. This implies that the equivalences (i)-(iii) given above can be reformulated only in terms of $H$ and $K$. In order to deal with (ii), we need the following notation.

Given an irreducible unitary  representation $\pi$ of $H$, let $K_\pi$ be the stabilizer of $[\pi]\in\widehat H$ under the action $k:\pi\longmapsto \pi\circ k\inv$ of $K$. Let also $\sigma_\pi$ be the associated projective representation of $K_\pi$, i.e., such that $\pi(k\inv\cdot h)=\sigma_\pi(k)\pi(h)\sigma_\pi(k)\inv$ for all $k\in K$, $h\in H$.

\begin{proposition}\label{KltimesH}
Let $H$, $K$, $G=K\ltimes H$ be as above.
Each of the following is equivalent to the condition that $(G,K)$ is a Gelfand pair:
\begin{enumerate}[\rm(i')]
\item $L^1(H)^K$ is commutative;
\item for every $\pi\in\widehat H$, $\sigma_\pi$ decomposes without multiplicities.
\end{enumerate}

If $H$ is connected, these conditions are also equivalent to
\begin{enumerate}
\item[\rm(iii')] $\bD(H)^K$ is commutative.
\end{enumerate}
\end{proposition}

\subsection{Strong Gelfand pairs and $K$-types}
\quad

The pair $(G,K)$ is called a {\it strong Gelfand pair} if  the $K$-central convolution algebra $L^1(G)^{{\rm Int}(K)}$ is commutative. By \eqref{inclusion}, a strong Gelfand pair is a Gelfand pair.

The structure of strong Gelfand pairs is made more precise by introducing the notion of $K$-type for functions on $G$.

For $\tau\in\widehat K$, a function $f\in L^1_{\rm loc}(G)$ is called {\it of $K$-type $\tau$} if the following identity holds for every $x\in G$:
\be\label{typetau}
f*_K (d_\tau \overline{\chi_\tau})(x)=d_\tau\int_Kf(xk\inv)\overline{\chi_\tau(k)}\,dk=f(x)\ ,
\ee
where $\chi_\tau$ and $d_\tau$ are character and dimension of $\tau$, respectively. Every $f\in L^1(G)$ decomposes into $K$-types,
\be\label{f->ftau}
f\sim \sum_{\tau\in\widehat K}f_\tau\ ,\qquad\big( f_\tau= f*_K(d_\tau\overline{\chi_\tau})\ \big)\ .
\ee

If $f$ is $K$-central, also its $K$-types $f_\tau$ are $K$-central, and the space $L^1_\tau(G)^{{\rm Int}(K)}$ of $K$-central functions of type $\tau$ is a convolution algebra. Moreover, the convolution of two $K$-central functions of different $K$-types is zero. In particular, for $f,g\in L^1(G)^{{\rm Int}(K)}$,
$$
f*g\sim\sum_{\tau\in\widehat K}f_\tau*g_\tau\ .
$$

Let now $V_\tau$ be the representation space of (some representative of) $\tau\in\widehat K$ and ${\rm End}(V_\tau)$ the algebra of endomorphisms of $V_\tau$.
There is  another realization of the $\tau$-type $K$-central algebra that we now describe \cite{Camp, RS, T, War}. 

For $f\in L^1_\tau(G)^{{\rm Int}(K)}$, we set
$$
A_\tau f(x)=\int_Kf(xk)\tau(k)\,dk\in L^1\big(G, {\rm End}(V_\tau)\big)\overset{\rm def}=L^1(G)\otimes {\rm End}(V_\tau)\ .
$$

Then 
\be\label{equivariance}
A_\tau f(k_1xk_2)=\tau(k_2\inv)A_\tau f(x)\tau(k_1\inv)\ ,
\ee
and we express this equivariance condition by writing $A_\tau f\in L^1_{\tau,\tau}\big(G, {\rm End}(V_\tau)\big)$.

Conversely, if $F\in  L^1_{\tau,\tau}\big(G, {\rm End}(V_\tau)\big)$, the function
$$
S_\tau F(x)=d_\tau\tr F(x)
$$
is in $L^1_\tau(G)^{{\rm Int}(K)}$ and the two maps $A_\tau$, $S_\tau$ are inverse of each other. Moreover they establish algebra isomorphisms, if convolution on $L^1_{\tau,\tau}\big(G, {\rm End}(V_\tau)\big)$ is defined as
$$
F_1*F_2(x)=\int_G F_1(y\inv x)F_2(y)\,dy\ .
$$

\begin{proposition}\label{eqstrong}
The following conditions are equivalent:
\begin{enumerate}[\rm(a)]
\item $(G,K)$ is  a strong Gelfand pair;
\item the pair  $(K\times G,{\rm diag}K)\cong\big({\rm Int}(K)\ltimes G,{\rm Int}(K)\big)$ is a Gelfand pair;
\item for every $\pi\in\widehat G$, $\pi_{|_K}$ decomposes into irriducibles without multiplicities;
\item for every $\tau\in\widehat K$, $L^1_\tau(G)^{{\rm Int}(K)}$ is commutative;
\item for every $\tau\in\widehat K$, $L^1_{\tau,\tau}\big(G, {\rm End}(V_\tau)\big)$ is commutative.
\medskip
\\
If $G$ is connected, these conditions are also equivalent to 
\smallskip

\item the algebra $\bD(G)^{{\rm Int}(K)}$ is commutative.
\end{enumerate}
\end{proposition}

Notice that, denoting by $\tau_0$ the trivial representation of $K$, $L^1_{\tau_0}(G)^{{\rm Int}(K)}=L^1(K\backslash G/K)$. 
\medskip

\section{Spherical functions and spectra}

From now on, we only consider Gelfand pairs $(G,K)$ with $G/K$ connected and, consequently, only strong Gelfand pairs $(G,K)$ with $G$ connected.

The {\it spherical functions} on $G$ are the $K$-invariant functions on $G/K$ that are
joint eigenfunctions $\ph$ of all elements of~$\bD(G/K)$ and have $\ph(eK)=1$. We use the same term, and the same symbol $\ph$, for the lifted bi-$K$-invariant function on $G$.

The set $\Sigma$ of {\it bounded} spherical functions, endowed with the compact-open topology, represents the Gelfand spectrum of $L^1(K\backslash G/K)$, also called the spectrum of $(G,K)$. Then every nontrivial multiplicative functional on $L^1(K\backslash G/K)$ can be expressed as
\be\label{spherical}
f\longmapsto \int_Gf(x)\ph(x\inv)\,dx\overset{\rm def}=\cG f(\ph)\ .
\ee

The operator $\cG: L^1(K\backslash G/K)\longrightarrow C_0(\Sigma)$ is the  {\it spherical transform} of $(G,K)$.

Calling, for $\ph$ spherical, $\la_\ph(D)$ the eigenvalue of $\ph$ under the action of $D\in\bD(G/K)$, the algebra homomorphism $\xi_\ph:\bD(G/K)\longrightarrow\bC$ identifies $\ph$ uniquely. Since $\bD(G/K)$ is finitely generated \cite[Cor. 2.8, Thm. 5.6]{H}, we can identify spherical functions with finite tuples of eigenvalues.

\begin{theorem} [\text{\cite{FR}}]\label{FerRuf}
Let $\cD=\{D_1,\dots,D_\ell\}$ be a finite generating system in $\bD(G/K)$. 
The map
\be\label{iota}
\iota_\cD(\ph)=\big(\xi_\ph(D_1),\dots,\xi_\ph(D_\ell)\big)\ .
\ee
 establishes a homeomorphism between $\Sigma$ and its image $\iota_\cD(\Sigma)=\Sigma_\cD\subset \bC^\ell$. Moreover, $\Sigma_\cD$ is closed in $\bC^\ell$.
\end{theorem}

If $G=K\ltimes H$, we can appeal to \eqref{L^1(H)} and \eqref{D(H)} and regard spherical functions as $K$-invariant functions on $H$, which are eigenfunctions of the operators in $\bD(H)^K$. 
\medskip

 Something more must be added in the case of a strong Gelfand pair $(G,K)$, recalling from Proposition \ref{eqstrong} that this means that $\big({\rm Int}(K)\ltimes G,{\rm Int}(K)\big)$ is a Gelfand pair. Coherently with the previous remark, we regard  spherical functions as $K$-central functions on $G$, which are eigenfunctions of the operators in $\bD(G)^{{\rm Int}(K)}$. Each spherical function has a unique $K$-type \cite[Prop. 7.3]{RS}, and this determines a partition 
\be\label{partition}
\Sigma=\bigcup_{\tau\in\widehat K}\Sigma^\tau
\ee 
of the Gelfand spectrum according to $K$-types, where each $\Sigma^\tau$ is the Gelfand spectrum of the $K$-type subalgebra $L^1_\tau(G)^{{\rm Int}(K)}$.

Notice that, for the trivial representation $\tau_0$, the component $\Sigma^0$ is the spectrum of the underlying (non-strong) Gelfand pair $(G,K)$.

\medskip

\section{Pairs with polynomial growth}

It is not possible, in general, to find generators such that the embedding \eqref{iota} is in $\bR^\ell$. This is the case, for instance, of symmetric pairs of the noncompact type, where the bounded spherical functions, as well as their eigenvalues, depend holomorphically on complex parameters \cite[Ch. 4]{H2}.

On the contrary, embeddings in real Euclidean spaces are possible when the group $G$ has {\it polynomial (volume) growth}, i.e., if there is $k\in \bN$ such that, given a compact neighborhood  $U$ of the identity, the Haar measures  $m(U^n)$ of powers of $U$ are $O(n^k)$ as $n\to\infty$. 
We say that the pair $(G,K)$ has polynomial growth if $G$ has polynomial growth. 

 The following result extends Theorem G in \cite{BJR1} to general Gelfand pairs of polynomial growth.

\begin{lemma}\label{xireal}
Let $(G,K)$ be a Gelfand pair with polynomial growth. All bounded spherical functions are positive definite. Let $\cD=\{D_1,\dots,D_\ell\}$ be a generating system of $\bD(G/K)$, with all the $D_j$  symmetric operators. Then $\Sigma_\cD\subset \bR^\ell$. 
\end{lemma}

\begin{proof}
Let $\ph$ be a bounded spherical function and 
$$
\Lambda_\ph(f)=\int_G f(x)\ph(x)\,dx
$$
the associated multiplicative functional  on $L^1(K\backslash G/K)$. 
Since $L^1(G)$ is symmetric \cite{L},
applying \cite[Thm. 1, p. 305]{N}, we can extend $\Lambda_\ph$ to an irreducible $*$-representation $\pi$ of the whole algebra $L^1(G)$. In formulae, there is a one dimensional subspace $V$ of the representation space $\mathcal{H}_\pi$ such that 
\[
\pi(f)v  =\Lambda_\ph(f)v \qquad \forall\ v\in V\ ,\ f\in L^1(K\backslash G/K).
\]

Denoting by $\pi_0$ the irreducible unitary representation of the group $G$ on $\mathcal{H}_\pi$ such that
\[
\pi(f)=\int_G f(x)\, \pi_0(x)\, dx\ ,
\]
it follows that, if $v$ is a unit vector in $V$,
\[
\int_G f(x)\,\ph(x\inv)\, dx=\int_G f(x)\, \langle \pi_0(x)v,v\rangle\, dx\ ,\ \forall f\in L^1(K\backslash G/K)\ ,
\]
therefore $\ph$ can be expressed as $\ph(x)=\lan v,\pi_0(x)v\ran$ and $\ph$ is positive definite.
 
 The bi-$K$-invariance of $\ph$ implies that $v\in \cH_{\pi_0}^K$, the subspace of $K$-fixed vectors, which is one-dimensional since $(G,K)$ is a Gelfand pair.  Then $\cH_{\pi_0}^K$ consists of $C^\infty$-vectors. Since  $d\pi_0(D_j)$ preserves  $\cH_{\pi_0}^K$ for each $j$ and is symmetric, we have $d\pi_0(D_j)v=\xi_j v$, with $\xi_j\in\bR$. Il follows that
 $$
 D_j\ph(x)=\lan v,\pi_0(x)d\pi_0(D_j)v\ran=\xi_j\ph(x)\ .\qquad\qedhere
 $$
\end{proof}

We call $\Sigma_\cD$ an {\it embedded copy of $\Sigma$}. Notice that the dimension $\ell$ of the ambient space depends on $\cD$ and can be rather arbitrary, since we are not imposing any minimality condition on $\cD$. The next lemma justifies that the choice of the embedding will be irrelevant for our purposes. The proof is in \cite{ADR2}.

\begin{proposition}\label{D,tildeD}
Let $\cD=\{D_1,\dots,D_\ell\}$ and $\tilde\cD=\{\tilde D_1,\dots,\tilde D_{\tilde\ell}\}$ be two systems of symmetric generators of $\bD(G)^{{\rm Int}(K)}$. Let $p_1,\dots,p_{\tilde\ell}$ be polynomials on $\bR^\ell$, $q_1,\dots,q_\ell$ polynomials on $\bR^{\tilde\ell}$ such that
$$
\tilde D_j=p_j(D_1,\dots,D_\ell)\ ,\quad(j=1,\dots,\tilde\ell)\ ,\qquad D_j=q_j(\tilde D_1,\dots,\tilde D_{\tilde\ell})\ ,\quad(j=1,\dots,\ell)\ .
$$

Setting $P=(p_1,\dots,p_{\tilde\ell})$, $Q=(q_1,\dots,q_\ell)$,  the two maps 
\be\label{PQ}
 {\Sigma_{\cD}\ }\overset P{\underset Q{\myrightleftarrows{\rule{1.5cm}{0cm}}}} \ \Sigma_{\tilde\cD}   
 \ee
 are mutually inverse homeomorphisms. Moreover, there is an exponent $m\ge1$ such that, for $\xi\in\Sigma_\cD$, $\xi'=P(\xi)\in\Sigma_{\tilde\cD}$, then
 \be\label{PQestimates}
\big(1+ |\xi|\big)^{1/m}\lesssim 1+|\xi'|\lesssim \big(1+ |\xi|\big)^m\ .
\ee
 \end{proposition}

The space $\cS(\Sigma _\cD)$ has been defined in \eqref{Ssigma}. The next statement implies that validity of  property (S) does not depend on the choice of the embedded copy of~$\Sigma$.

\begin{corollary}[{\cite{ADR2}}]\label{Dindependence}
The maps $P,Q$ in \eqref{PQ} induce isomorphisms between $\cS(\Sigma_\cD)$ and $\cS(\Sigma_{\tilde\cD})$.  
\end{corollary}

The next statement shows the r\^ole of $\Sigma_\cD$ in the spectral analysis of the operators in $\cD$ (By \cite{Ne}, the $D_j$, that we have chosen to be symmetric, are essentially self-adjoint).

\begin{proposition}\label{jointspectrum}
The set $\Sigma_\cD$ is the joint spectrum of the operators $D_1,\dots,D_\ell$ as self-adjoint operators on $L^2(G/K)$.
\end{proposition} 

\begin{proof}
 By the Plancherel-Godement theorem \cite{Far, vanD}, the spherical transform $\cG$ is a unitary operator from $L^2(K\backslash G/K)$ onto $L^2(\Sigma,\beta)$, where $\beta$ is the Plancherel measure on $\Sigma$. By symmetry of $L^1(K\backslash G/K)$, $\beta$ is supported on all of $\Sigma$. Given a generating system $\cD=\{D_1,\dots,D_\ell\}$ of $\bD(G/K)$, we call $\beta_\cD$ the Radon measure on $\bR^\ell$, with support equal to $\Sigma_\cD$, which is the push-forward of $\beta$ by $\iota_\cD$. 
 
 Since the operators $D_j$ preserve the subspace $L^2(K\backslash G/K)$, they can be regarded as self-adjoint operators on $L^2(K\backslash G/K)$. 
  The spherical transform $\cG$ intertwines $D_j$ with the multiplication operator $\widetilde D_j$ on $L^2(\Sigma_\cD,\beta_\cD)$ given by $\widetilde D_jg(\xi)=\xi_jg(\xi)$. The joint spectrum of the $\widetilde D_j$ is the support of $\beta_\cD$, i.e., $\Sigma_\cD$.
  
  This identifies $\Sigma_\cD$ as the joint spectrum of the $D_j$ in the algebra $\cL\big(L^2(K\backslash G/K)\big)$. The spectral measure $E(A)$ of a Borel set $A\subset\Sigma_\cD$ of finite $\beta_\cD$-measure is the convolution operator
  $$
  E(A)f=f*\cG\inv \chi_A\ .
  $$
  
  But these operators are defined on all of $L^2(G/K)$ and define a resolution of the identity on it based on $\Sigma_\cD$ and such that
  $$
  D_j=\int_{\Sigma_\cD}\xi_j\,dE(\xi)\ ,
  $$
  for all $j$. This concludes the proof.
\end{proof}

\bigskip

\section{Property (S) on pairs of polynomial growth}

We begin by defining the Schwartz space $\cS(G)$ on a connected Lie group $G$ with polynomial  growth. 

For a fixed compact, symmetric neighborhood $U$ of $e$, define\footnote{The choice of $U$ is irrelevant. Moreover, $|x|$ can be replaced by $d(x,e)$ for any left-invariant  quasi-distance $d$ on $G$, ``connected'' in the sense of \cite{VSC} and compatible with its topology. This is the case, e.g., for the Carnot-Carath\'eodory distance induced by any H\"ormander system in $\fg$.}
$$
|x|=\begin{cases} 0&\text{ if }x=e\\ \min\{n:x\in U^n\}&\text{ if }x\neq e\ .\end{cases}
$$

Let $\{X_1,\dots,X_n\}$ be an ordered linear basis of  $\fg$ and set, for every multiindex $\al=(\al_1,\dots,\al_n)$,
$$
X^\al=X_1^{\al_1}X_2^{\al_2}\cdots X_n^{\al_n}\ .
$$

For $N\in\bN$, the $N$-th Schwartz norm on $\cS(G)$ is given by
$$
\|f\|_{(N)}=\sup_{|\al|\le N}\sup_{x\in G}\big(1+|x|\big)^N\big|X^\al f(x)\big|\ .
$$
\medskip

In connection with property (S), we recall the following general result from \cite{Mar1}.

 \begin{proposition}[\cite{Mar1}, Prop. 4.2.1]
 Let $\cD=\{D_1,\dots,D_\ell\}$ be a weighted subcoercive family of self-adjoint, strongly commuting, left-invariant differential operators of a Lie group $G$ with polynomial growth. Given $m\in\cS(\bR^\ell)$, the operator $m(D_1,\dots,D_\ell)$ has a Schwartz convolution kernel $K=\cK_\cD m$. Moreover, the operator $\cK_\cD:\cS(\bR^\ell)\longrightarrow \cS(G)$ is continuous.
  \end{proposition}

By Proposition \ref{jointspectrum}, this implies that $\cG\inv$ maps $\cS(\Sigma_\cD)$ continuously into $\cS(K\backslash G/K)$, see also \cite{ADR2} for the case of nilpotent pairs.
To prove the corresponding statement for $\cG$ is a more challenging problem. The open mapping theorem (cf. \cite{Tr} for its version on Fr\'echet spaces) allows us to disregard continuity of $\cG$  and simply prove that it maps $\cS(K\backslash G/K)$  into $\cS(\Sigma_\cD)$.

\begin{proposition}\label{S'}
Property {\rm(S)} is satisfied if and only if 
\begin{enumerate}
\item[\rm(S')] for every $f\in \cS(K\backslash G/K)$, $\cG f$ extends to a Schwartz function on $\bR^\ell$.
\end{enumerate}
\end{proposition}

\bigskip

\section{Strong Gelfand pairs and embeddings of their spectra}

For a strong Gelfand pair $(G,K)$, not necessarily of polynomial growth, we have seen that the $K$-central algebra $L^1(G)^{{\rm Int}(K)}$ splits as the sum of the subalgebras (in fact ideals) $L^1_\tau(G)^{{\rm Int}(K)}$, with $\tau\in\widehat K$, and that its spectrum $\Sigma$ is the disjoint union of their spectra $\Sigma^\tau$, cf.  \eqref{partition}.

We will coherently denote by $\cG f\in C_0(\Sigma)$ the spherical transform of  $f\in L^1(G)^{{\rm Int}(K)}$ and by $\cG_\tau f\in C_0(\Sigma^\tau)$ the spherical transform of  $f\in L^1_\tau(G)^{{\rm Int}(K)}$, so that, taking into account the Fourier decomposition of $f$ in \eqref{f->ftau},
$$
(\cG f)_{|_{\Sigma^\tau}}=\cG_\tau f_\tau\ .
$$

In particular, $\cG f$ is supported on $\Sigma^\tau$ if and only if $f\in L^1_\tau(G)^{{\rm Int}(K)}$.

In the algebra $\bD(G)^{{\rm Int}(K)}\cong\fU(\fg)^{{\rm Ad}(K)}$ we now choose  a generating system $\cD=(D_1,\dots, D_\ell)$ where each $D_j$ is symmetric and $D_1,\dots, D_r$, $r<\ell$, generate 
$\fU(\fk)^{{\rm Ad}(K)}$, the center of $\fU(\fk)$. With a language abuse, for $j\le r$ we use the same symbol $D_j$ to denote the differential operator on  both $G$ and  $K$.

 For $f$ of type $\tau$ and $j\le r$, we have
$$
D_j f=d_\tau D_j(f*_K\overline{\chi_\tau})=d_\tau f*_K(D_j\overline{\chi_\tau})\ ,
$$
so that, recalling the notation
$$
\iota_\cD(\ph)=(\xi_1,\dots,\xi_\ell)\ ,\quad \text{if }\ D_j\ph=\xi_j\ph\ , j=1,\dots,\ell\ ,
$$
 the first $r$ entries of $\iota_\cD(\ph)$ are common to all spherical functions of the same type and determine the type uniquely.

\medskip

Assuming now $G$ of polynomial growth, 
we decompose $\xi\in\bR^\ell$ as $(\xi',\xi'')\in\bR^r\times\bR^{\ell-r}$ (preserving the order of coordinates), and set 
\be\label{Sigmatau}
\iota_\cD=(\iota'_\cD,\iota''_\cD)\ ,\qquad \xi'_\tau=\iota'_\cD(\ph)\quad(\ph\in\Sigma^\tau)\ ,\qquad\Sigma^\tau_\cD=\iota''(\Sigma^\tau)\ .
\ee
where the components of $\xi'_\tau$ are the eigenvalues of the character $\chi_\tau$ under $D_1,\dots,D_r$. Then
$$
\Sigma_\cD=\bigcup_{\tau\in\widehat K}\{\xi'_\tau\}\times\Sigma^\tau_\cD\ .
$$

\section{Reduction of {\rm(S')} to $K$-types}

In this section we prove the equivalence between condition (S') in Proposition \ref{S'} and a weaker form of it, which takes into account the decomposition into $K$-types. 

\begin{theorem}\label{reduction}
Property (S') holds for a strong Gelfand pair $(G,K)$ of polynomial growth if and only if the following condition is satisfied:
\begin{enumerate}
\item[\rm(S'')] given $f\in \cS(G)^{{\rm Int}(K)}$ and $N\in\bN$, for each component $f_\tau\in \cS_\tau(G)^{{\rm Int}(K)}$ of $f$,   $\cG_\tau f_\tau$ admits a Schwartz extension $g_\tau^N$ such that $\|g_\tau^N\|_{(N)}$ is rapidly decaying in $\tau$.
\end{enumerate}
\end{theorem}

Before giving the proof, and in order to clarify the meaning of ``rapid decay in $\tau$'', we must say more about the set
$$
\Sigma'_\cD=\{\xi'_\tau:\tau\in\widehat K\}\subset\bR^r\ .
$$

We denote by 
 $\kappa_\tau$ the infinitesimal character of $\tau\in\widehat K$, i.e., the multiplicative functional on $\fU(\fk)^{{\rm Ad}(K)}\cong \bD(K)^{{\rm Int}(K)}$ such that, for $D\in  \bD(K)^{{\rm Int}(K)}$ (i.e., in the center of $\bD(K)$),

$$
 d\tau(D)=\kappa_\tau(D)I\ .
$$

We fix  a maximal toral subalgebra $\ft$ of $\fk$, denote the Weyl group of $(\fk,\ft)$ by $W$ and  fix a positive system of roots. We also call~$\rho$ half the sum  of the positive roots and $\mu_\tau$ the highest weight of $\tau$.

If $\gamma: \bD(K)^{{\rm Int}(K)}\longrightarrow \cP(\ft^*)^W$ is the Harish-Chandra isomomorphism \cite[Thm. 4.10.3]{V}, then 
$$
\kappa_\tau(D)=\big(\gamma(D)\big)(\mu_\tau+\rho)\ .
$$

In particular, setting $p_j= \gamma(D_j)$ for $j=1,\dots, r$, the polynomials $p_1,\dots,p_r$  generate $\cP(\ft^*)^W$ and  we have
\be\label{xi'=p}
\xi'_\tau=\big(p_1(\mu_\tau+\rho),\dots, p_r(\mu_\tau+\rho)\big) \ .
\ee

Consider the map $P=(p_1,\dots,p_r)$ from $\ft^*$ to $\bR^r$. We may assume, by Proposition \ref{D,tildeD}, that one of the polynomials $p_j(t)=|t|^2$ where $|\ |$ is a Weyl-invariant Hilbert norm on $\ft^*$. 

Since $P$ is a homeomorphism of $\ft^*/W$ onto $P(\ft^*)$ \cite{GS},  we can conclude that
\be\label{xi'/tau}
C|\mu_\tau+\rho|^2\le 1+|\xi'_\tau|\le C'|\mu_\tau+\rho|^{d} \ ,
\ee
for some $d\ge2$.
Therefore, rapid decay in $\tau$ can be equivalently expressed as rapid decay w.r. to $|\xi'_\tau|$ or to $|\mu_\tau|$.

\begin{lemma}\label{interpolation}
There are sequences $\{M(N)\}$, $\{A_N\}$ such that  any rapidly decreasing function $a$ on $\Sigma'_\cD$ can be interpolated by a Schwartz function $u\in\cS(\bR^r)$ such that, for every $N$,
\be\label{a->g}
\|u\|_{(N)}=\sup_{\xi'\in\bR^r}\sup_{|\al|\le N}\big(1+|\xi'|\big)^N\big|\de^\al u(\xi')\big|\le A_N\sup_\tau\big(1+|\xi'_\tau|\big)^{M(N)}\big|a(\xi'_\tau)\big|\ .
\ee
\end{lemma}

\begin{proof}

We can set $a(\xi'_\tau)=\tilde a(\mu_\tau+\rho)$ and extend $\tilde a$ to a Weyl-invariant function on the integral lattice $\Lambda$ in $\ft^*$. Then also $\tilde a$ is rapidly decreasing. If $\eta\in C^\infty_c(\ft^*)$ is Weyl-invariant, equal to 1 at 0 and supported on a small neighborhood of 0 which does not contain any other element of $\Lambda$, then
$$
h(t)=\sum_{\la\in\Lambda}\tilde a(\la)\eta(t-\la)
\qquad\forall t\in \ft^*
$$
extends $\tilde a$, is Schwarz and Weyl-invariant, and the $N$-th Schwartz norm of $h$ is controlled by the right-hand side of \eqref{a->g} for some $M(N),A_N$.

 We apply G. Schwartz's extension of Whitney's theorem \cite{GS} and use the existence of an extension operator $\cE$ \cite{Mather}
 \footnote{See the adaptation to Schwartz functions in \cite{ADR2}.}. This means that there is $g=\cE h\in\cS(\bR^r)$ such that
 $$
 h(t)=u\big(p_1(t),\dots,p_r(t)\big)\ ,
 $$
 where the operator $\cE$ is linear and continuous, i.e., there exist $M'(N),A'_N$ such that
 $$
 \|u\|_{(N)}\le A'_N\|h\|_{(M'(N))}
 $$
 for every $N$.
This gives \eqref{a->g}. Moreover, 
\begin{equation*}
 a(\xi'_\tau)=\tilde a(\mu_\tau+\rho)=h(\mu_\tau+\rho)=g(\xi'_\tau)\ .
  \qedhere
\end{equation*}
\end{proof}

\begin{proof}[Proof of Theorem \ref{reduction}] 
Assume that property (S') holds for $(G,K)$ and 
fix $\tau$ in $\widehat K$. 

If $g$ is a Schwartz extension of $\cG f$ to $\bR^\ell$, then $g_\tau(\xi'')=g(\xi'_\tau,\xi'')$ is a Schwartz extension to $\bR^{\ell-r}$ and, for every $M, N$,
\be\label{decay}
\|g_\tau\|_{(N)}\le C_{M,N}\big(1+|\xi'_\tau|\big)^{-M}\ ,
\ee
and this gives (S''), in fact with the same $g^\tau$ independent of $M$.

Assume now that (S'') holds. By a diagonal argument, it is possible to produce a single sequence of functions $g_\tau=g^{N(\tau)}_\tau$ such that \eqref{decay} holds for every $M,N$. In fact  there exists an increasing sequence  $\{r_N\}$ such that, for $|\xi'_\tau|>r_N$, $\|g_\tau^N\|_{(N)}\le |\xi'_\tau|^{-N}$. We then choose
$$
N(\tau)=\begin{cases}0&\text{ if }|\xi'_\tau|\le r_0\\ N&\text{ if }r_N<|\xi'_\tau|\le r_{N+1}\ .\end{cases}
$$

Applying Lemma \ref{interpolation} to each function $a_{\xi''}(\xi'_\tau)=g_\tau(\xi'')$ on $\Sigma'_\cD$ and calling $u_{\xi''}$ the resulting Schwartz function on $\bR^r$, it is easy to verify that the function
$$
g(\xi)=u_{\xi''}(\xi')
$$
is Schwartz.
\end{proof}

\medskip

We conclude this section with the proof of property (S) for compact Gelfand pairs, i.e., with $G$ compact, which is based on a factorization of the space $\bR^\ell$ containing the spectrum $\Sigma$ similar to the one used for strong Gelfand pairs.
We refer to \cite[Ch. VI, Sect. 2.1]{H}  for notation and more details.

\begin{theorem}\label{compactGp}
Property (S) holds for compact Gelfand pairs.
\end{theorem}

\begin{proof}
Let $G$ be compact, $K\subset G$. Let $\widehat G_K$ the set of elements of $\widehat G$ of class 1 relative to $K$. By the Peter-Weyl theorem,
$$
L^2(G/K)=\sum_{\pi\in\widehat G_K}\cH_\pi(G/K)\ ,
$$
where, up to lifting functions to $G$, $\cH_\pi(G/K)\subset\cH_\pi(G)$ is the space of  right-$K$-invariant matrix entries of $\pi$. Then each $\cH_\pi(G/K)$ contains a single spherical function.

The algebra $\bD(G/K)$ contains the algebra $\cA=\bZ(G/K)$ of operators induced from the elements of the center $\bZ(G)$ of $\bD(G)$.  We can then choose a generating system $\cD$ of $\bD(G/K)$ by first selecting a generating system of $\bZ(G)$. Projecting these operators to $G/K$ we obtain a generating system $\{D_1,\dots,D_r\}$ of $\cA$, which can be completed, if necessary, to form a generating system $\{D_1,\dots,D_\ell\}$ of $\bD(G/K)$.

Since the (scalar) actions of $\bZ(G)$ on the various $\cH_\pi(G)$ are inequivalent, the eigenvalues under the action of $D_1,\dots,D_r$ are sufficient to identify spherical functions. This means that the projection of $\Sigma_\cD\subset\bR^\ell$ into $\bR^r$ is injective and consists of the points \eqref{xi'=p}  with $\tau\in \widehat G_K$, so that inequality \eqref{xi'/tau} holds. We can then exploit the arguments used in the first part of this section to obtain the required estimates.
\end{proof}

\bigskip

\section{A first step in the study of property {\rm (S)} for strong Gelfand pairs\\ with polynomial growth}

From \cite[Thm. 3]{Yak} or \cite[Thm. 15.1.6]{W} one derives the list of all strong Gelfand pairs $(G,K)$ which are indecomposable and maximal.\footnote{\cite{Yak} and \cite{W} do not mention strong Gelfand pairs explicitly. One must recall that   $(G,K)$ being a strong Gelfand pair is the same as saying that $(G\times K,{\rm diag}\,K)$ is a Gelfand pair.}  Indecomposable means that it cannot be expressed as $(G_1\times G_2,K_1\times K_2)$ with $K_1\subset G_1$, $K_2\subset G_2$.
Maximal means that $G$ is not obtained from another strong Gelfand pair $(\widetilde G,K)$ by central reduction, i.e.,  $G=\widetilde G/H$, where $H\subset Z(G)$, the center of $G$, and $(\widetilde G,K)$ is a strong Gelfand pair.

They are the following:

  \begin{table}[htp]
\begin{center}
\begin{tabular}{lcll}
Compact&\quad&$(SO_{n+1},SO_n)$&$(U_{n+1},U_1\times U_n)$\\
Polynomial growth, noncompact&\quad&$(SO_n\ltimes\bR^n,SO_n)$&$(U_n\ltimes H_n,U_n)$\\
Exponential growth&\quad&$(SO_{n,1},SO_n)$&$(U_{n,1},U_1\times U_n)$\\
\end{tabular}
\end{center}
\end{table}

 In view of Theorem \ref{compactGp}, we restrict ourselves to pairs that are  noncompact and of polynomial growth.  Besides the pairs in the second row, there are non-maximal ones that are also interesting, like 
 $$
 (U_n\ltimes \bC^n,U_n)\ ,\qquad (U_1^n\ltimes H_n,U_1^n)\ .
 $$

They are both obtained by central reduction, the first from $(U_n\ltimes H_n,U_n)$, factoring out the center of $H_n$, the second from the $n$-fold product of $(U_1\ltimes H_1,U_1)$, factoring out a hyperplane in the center of $(H_1)^n$.

Concerning property (S), we prove here the following theorem.

\begin{theorem}
 Let $(G,K)$ be a Gelfand pair with $G=K\ltimes H$ and $K$ abelian.  Then $(G,K)$ is automatically a strong Gelfand pair. Moreover, assuming that $G$ has polynomial growth, $(G,K)$ satisfies property (S) as  strong Gelfand pair if and only if it does so as  Gelfand pair.
 \end{theorem}
 
 We recall the difference between the two kinds of property (S) for $(G,K)$: as strong Gelfang pair, it concerns Schwartz extendability of $\cG f$, with $f$ $K$-central on $G$, from $\Sigma_\cD$ to $\bR^\ell$, while, as ordinary Gelfand pair, it concerns extendability of $\cG_0 f$, with $f$ bi-$K$-invariant on $G$, from $\Sigma^{\tau_0}_\cD$ to $\bR^{\ell-r}$.
 
 \begin{proof}
We denote the elements of $K$ as $k_\theta$ with $\theta \in\bT^r$ and choose $D_j$, $j\le r$, as $D_jf(x)=-i\de_{\theta_j}f(xk_\theta)_{|_{\theta=0}}$. We then have $\Sigma'_\cD=\bZ^r(\cong \widehat K)$. 
 Writing the product on $K\ltimes H=G$ as in~\eqref{product},
 for each $m\in\bZ^r$, $L^1_m(G)^{{\rm Int}(K)}$ consists of the  functions $f$ such that, for all $x\in G$ and $\theta\in\bT^r$,
 \be\label{tau->0}
  f(k_\theta,h)= f_m(h)e^{im\cdot\theta}\ ,
\ee
with $f_m$ $K$-invariant on $H$.

We also have, for $f,g\in L^1_m(G)^{{\rm Int}(K)}$,
$$
(f*_Gg)(k_\theta,h)=( f_m*_H g_m)(h)e^{im\cdot\theta}\ ,
$$
so that
$$
L^1_m(G)^{{\rm Int}(K)}\cong L^1_0(G)^{{\rm Int}(K)}\cong L^1(H)^K
$$ as convolution algebras. This implies the first part of the statement. 

We also have the equality $\Sigma^m_\cD=\Sigma^0_\cD$ in $\bR^{\ell-r}$ and therefore the spectrum of the strong Gelfand pair is $\Sigma_\cD=\bZ^r\times\Sigma^0_\cD$.\footnote{Observe that this identity may not hold if $G$ is not a semidirect product $K\ltimes H$. Take for instance $(G,K)=(SO_3,SO_2)$ or $=(SO_{2,1},SO_2)$.}

Passing to the second part, one implication is implicit in condition (S''). For the other implication, we use Theorem \ref{reduction} as follows.

Assume that property (S) is satisfied by $(G,K)$ as a Gelfand pair and take $f\in \cS(G)^{{\rm Int}(K)}$. Then
$$
f(k_\theta,h)=\sum_{m\in\bZ^r} f_m(h)e^{im\cdot\theta}\ ,
$$
where the functions $ f_m$ are in $\cS(H)^K$ with each Schwartz norm rapidly decreasing in $m$.

By Proposition \ref{KltimesH}, we can state property (S) as a Gelfand pair  as continuity for the map $\cG_0:\cS(H)^K\longrightarrow \cS(\Sigma_\cD^0)$. Explicitly, this means that, for every $N$, there exist an integer $M=M(N)$ and a constant $C_N$ such that, for every $f\in \cS(H)^K$, $\cG_0 f$ admits a Schwartz extension $g_N$  with $\|g_n\|_{(N)}\le C_N\|f\|_{(M)}$.

Applying this to each $ f_m$, we obtain condition (S'').
\end{proof}

 In particular we have
\begin{corollary}
The pairs $(U_1\ltimes \bC,U_1)$, $(U_1^n\ltimes H_n, U_1^n)$ are strong Gelfand pairs satisfying property (S).
\end{corollary}

\begin{proof} Property (S) as Gelfand pairs has been proved for these pairs in \cite{ADR1,ADR2}.
\end{proof}

\medskip

It is clear that matters become more complicated if $K$ is not abelian. 
Though the group $G$ remains of the form $K\ltimes H$ (with $H$ equal to $\bR^n$, $\bC^n$ or a Heisenberg group $H_n$), the representations $\tau\in\widehat K$ mostly have dimension higher than 1. The hard points in our paradigm are the following two aspects of condition (S''), i.e.,
\begin{itemize}
\item for every $\tau$, $\cG^\tau_\cD$ maps $\cS(G)_\tau^{{\rm Int}(G)}$ into $\cS(\Sigma^\tau_\cD)$,
\item the rapid decay in $\tau$ prescribed by (S'') can be obtained.
\end{itemize}

This is done in \cite{ADR3} for the pair $(U_2\ltimes\bC^2,U_2)$.

\bigskip

\end{document}